\documentclass{amsproc}
\usepackage{amsmath,amsfonts,amssymb,amsthm,amscd,mathdots,bbm,graphicx,xcolor}
\usepackage{comment}
\usepackage{eucal}
\usepackage{mathrsfs}  
\numberwithin{equation}{section}

\def\R{{\mathbb R}}

\def\C{{\mathbb C}}
\def\H{{\mathbb H}}

\def\i{{\mathbf i}}
\def\j{{\mathbf j}}
\def\k{{\mathbf k}}

\def\t{\theta}

\newcommand{\be}{\begin{equation}}
\newcommand{\ee}{\end{equation}}

\renewcommand{\P}{\mathbb{P}}
\newcommand{\E}{\mathbb{E}}

\newcommand{\Conf}{\operatorname{Conf}}
\newcommand{\Pf}{\operatorname{Pf}}

\newcommand{\qdet}{\operatorname{Qdet}}

\newtheorem{thm}{Theorem}
\newtheorem*{claim}{Claim}
\newtheorem{cor}[thm]{Corollary}
\newtheorem{lem}[thm]{Lemma}
\newtheorem{prop}[thm]{Proposition}
\theoremstyle{definition}
\newtheorem{dfn}{Definition}
\newtheorem{ex}{Example}
\theoremstyle{remark}
\newtheorem{rmk}{Remark}
\numberwithin{thm}{section}
\numberwithin{rmk}{section}
\numberwithin{dfn}{section}

\begin{document}

\title[Conditional measures for Pfaffian point processes]{Conditional measures for Pfaffian point processes: conditioning on a bounded domain}

\author[A. I. Bufetov]{Alexander I. Bufetov}
\address[Alexander I. Bufetov]{\newline
Aix-Marseille Universit{\'e}, Centrale Marseille, CNRS, Institut de Math{\'e}matiques de Marseille, UMR7373, 
  39 Rue F. Joliot Curie 13453, Marseille, France;\newline
  Steklov  Mathematical Institute of RAS, Moscow, Russia;\newline
    Institute for Information Transmission Problems, Moscow, Russia.  }
  \email{alexander.bufetov@univ-amu.fr}

\author[F. D. Cunden]{Fabio Deelan Cunden}
\address[Fabio Deelan Cunden]{\newline
School of Mathematics and Statistics, University College Dublin, Dublin 4, Ireland.}
\email{fabio.cunden@ucd.ie}

\author[Y. Qiu]{Yanqi Qiu}
\address[Yanqi Qiu]{\newline
Institute of Mathematics, AMSS, CAS \and Hua Loo-Keng Key Laboratory of Mathematics, Beijing, China; \newline
 CNRS, Institut de Math{\'e}matiques de Toulouse, Universit{\'e} Paul Sabatier, Toulouse, France.}
\email{yanqi.qiu@hotmail.com, yanqi.qiu@amss.ac.cn}

\maketitle

\begin{abstract} 
For  a Pfaffian point process  we show that its Palm measures,  its normalised compositions with  multiplicative functionals, and its conditional measures with respect to fixing the configuration in  a bounded subset are Pfaffian point processes whose kernels we find explicitly. 
\end{abstract}

\section{Introduction and main results}
\subsection{Pfaffian point processes}
Let $E$ be a locally compact   $\sigma$-compact Polish space, endowed with a positive  $\sigma$-finite Radon measure $\mu$. We assume that the metric on $E$ is such that  any bounded set  is relatively compact, see Hocking and Young \cite[Theorem 2-61]{Hocking-Young}. A (locally finite) {\it configuration} $X$ on $E$ is a collection of points of $E$ (possibly with multiplicities and considered without regard to order) such that any bounded subset of $E$ contains only finitely many points of $X$.  A configuration $X$ on $E$ is called {\it simple} if all points in $X$ have multiplicity one.   Let $\Conf(E)$ denote the set of all configurations on $E$. By identifying  any configuration $X \in \Conf(E)$ with the  Radon measure $\sum_{x \in X} \delta_x$ on $E$, the set $\Conf(E)$ is embedded into the space $\mathfrak{M}(E)$ of Radon measures on $E$. In particular, with respect to the vague topology on $\mathfrak{M}(E)$, the subspace $\Conf(E)$ becomes a Polish space. We equip  $\Conf(E)$ with its Borel sigma-algebra.

By definition, a point process on $E$ is a Borel probability measure on $\Conf(E)$. A point process $\P$ on $E$ is called simple, if $\P$-almost every configuration is simple.  For further background on the general theory of point processes, see  Daley and Vere-Jones \cite{DV-1}, Kallenberg \cite{Kallenberg}.

Now we recall the definition of Pfaffian point processes. Recall that the {\it $k$-point correlation function} $\rho_{k}: E^k \rightarrow \R_{+}$ (with respect to the product measure $\mu^{\otimes k}$) of a point process $\P$, if  exists, is uniquely determined (up to $\mu^{\otimes k}$-negligible subsets) by 
$$
  \int\limits_{A_1^{l_1} \times  \cdots \times A_j^{l_j}} \rho_k(x_1, \cdots, x_k) d\mu^{\otimes k}(x_1, \cdots, x_k)  =  \int\limits_{\Conf(E)}   \prod_{i =1}^j  \frac{ \#(X \cap A_i)!}{( \# (X \cap A_i)  - l_i)! } d\P( X), 
$$
for all bounded disjoint Borel subsets $A_1, \cdots, A_j \subset E$ and positive integers $l_1, \cdots, l_j$ with $l_1 + \cdots + l_j = k$.  If $ \#(X \cap A_i)<l_i$, then we set $\#(X \cap A_i)!/( \#(X \cap A_i) - l_i)! = 0$. The measure 
\begin{align}\label{def-lambda-k}
d\lambda_k(x_1, \cdots, x_k)   = \rho_k(x_1, \cdots, x_k) d\mu^{\otimes k} (x_1, \cdots, x_k)
\end{align}
is called the {\it $k$-th correlation measure} of the point process $\P$.

Let $M_2(\C)$ denote the set of $2\times 2$ complex matrices. 
A  point process $\P$ on $E$ is called a {\it Pfaffian point process} if there exists a  matrix kernel $\mathbb{K}: E\times E \rightarrow M_2(\C)$: 
\[
\mathbb{K}(x, y) = \left[
\begin{array}{cc}
\mathbb{K}_{11}(x, y) & \mathbb{K}_{12}(x, y) 
\\
\mathbb{K}_{21}(x, y) & \mathbb{K}_{22}(x, y)
\end{array}
\right]
\]
satisfying 
\begin{align}\label{anti-sym}
\mathbb{K}(x, y)^t =  - \mathbb{K}(y, x),
\end{align}
such that for any $k\ge 1$, the $k$-th correlation function  $\rho_k(x_1, \cdots, x_k)$ of $\P$ exists and is given by 
\[
\rho_k(x_1, \cdots, x_k) = \Pf[\mathbb{K}(x_i, x_j)]_{1 \le i, j \le k}. 
\]
Here $\Pf(A)$ is the Pfaffian of an antisymmetric matrix and the condition \eqref{anti-sym} implies that the $2k \times 2k$ matrix $[\mathbb{K}(x_i, x_j)]_{1\le i, j \le k}$ is antisymmetric.  Note that Pfaffian point processes are simple point processes. 

In this paper, it is more convenient to work with an equivalent definition of Pfaffian point processes in terms of quaternion determinants of quaternion valued kernel. The Dyson-Moore definition of quaternion determinant is recalled in Appendix~\ref{app:qdet}. 

Let $\H_\C$ denote the complex algebra of {\it complexified quaternions} consists of elements $q = q_0+ q_1\i+ q_2\j +q_3 \k$, where $q_0, q_1, q_2, q_3\in \C$ are complex numbers and  $\i, \j, \k$ are the quaternion units with the rules 
$$
\i^2=\j^2=\k^2=-1,\quad \i\j=-\j\i=\k,\quad \j\k=-\k\j=\i,\quad \k\i=-\i\k=\j.
$$
The adjoint of $q = q_0+ q_1\i+ q_2\j +q_3 \k$ is defined as 
\[
\overline{q}= q_0 - q_1 \i -q_2 \j  - q_3 \k.
\]
 Note that we have 
$
\overline{qr}=\overline{r}\,\overline{q}.
$
 Moreover, the \emph{scalar product property} holds: 
$
pq+\overline{q}\,\overline{p}=qp+\overline{p}\,\overline{q}.
$
 The set of \emph{scalars} $q=\overline{q}$ coincides with the field $\C$ of complex numbers, which is canonically identified with a subfield in $\H_\C$ via embedding $\C\ni a\mapsto a + 0 \i + 0 \j + 0 \k\in\H_\C$.

A square quaternion matrix $M$ is called  \emph{self-adjoint} if $M_{ij}=\overline{M_{ji}}$ for all $i,j$; it is  called \emph{almost self-adjoint} \cite{Dyson72} if there is an integer $k$ such that
\be
M_{ij}=\overline{M_{ji}}\quad \text{ for $i\neq k$ and $j\neq k$}.
\label{eq:a-s-a}
\ee
In particular, all self-adjoint matrices are almost self-adjoint. 
If $M$ is an almost self-adjoint quaternion matrix , then we denote $\qdet M$  the Dyson-Moore determinant of $M$;  the definition and properties of $\qdet M$ are recalled in Appendix~\ref{app:qdet}. We need almost self-adjoint quaternion matrices since these appear in the study of Palm measures, cf. Remark \ref{rem-almost-sa} below. 

A point process $\P$ on $E$ is called a  Pfaffian point process if there exists 
a self-adjoint quaternion kernel $K\colon E\times E\to\H_\C$ such that for any positive integers $k$, the $k$-point correlation function of $\P$ exists and has the form
$$
\rho_{k}(x_1,\dots,x_k)=\qdet[K(x_i,x_j)]_{i,j=1}^n.
$$
In this case, $K$ is called a \emph{correlation kernel} of the Pfaffian point process $\P$ and  we denote the point process by $\P_{K}$. 

\begin{rmk}
For the equivalence of the two definitions of Pfaffian point processes, we refer to the equality \eqref{qdet-pf} in Appendix~\ref{app:qdet}. 
\end{rmk}

\begin{ex}[$\operatorname{CSE}$ process]\label{ex-1} Let $E=(-\pi,\pi]$. The probability density 
\[
p_{N}(\theta)=\frac{2}{(2\pi)^N(2N)!}\prod_{1\leq j<k\leq N}\left|e^{i\theta_j}-e^{i\theta_k}\right|^4,\quad\theta=(\t_1,\dots,\t_N)\in E^N,
\]
is the eigenvalue density of the Circular Symplectic Ensemble ($\operatorname{CSE}$) of random matrix theory.  This density defines a simple point process on $E$ with correlation functions
\[
\rho_{n}(\theta_1,\dots,\theta_n)=\frac{(N-n)!}{n!}\int\limits_{E^{N-n+1}} p_{N}(\theta)d\theta_{n+1}\cdots d\theta_{N}.
\]
It is well-known~\cite{Dyson70,Kargin14} that the point process is Pfaffian,
\[
\rho_{n}(\theta_1,\dots,\theta_n)=\qdet[K_{4,N}(\theta_i,\theta_j)]_{i,j=1}^n,
\]
 with translation invariant quaternion kernel $K_{4,N}(\t_1,\t_2)=\sigma_{4,N}(\theta_1-\theta_2)$:
\[
\sigma_{4,N}(\t)=\frac{1}{2\pi}\sum_{p=1/2}^{N-1/2}\left(\cos p\t+a_p\sin p\t\right),
\]
where $a_p=\dfrac{1}{2p}\left[ i (p^2-1) \i+(p^2+1)\j\right]$. The kernel can be also written as
\[
\sigma_{4,N}(\t)=\frac{1}{4}\Big[2s_{2N}(\t)- i (Is_{2N}(\t)+Ds_{2N}(\t))\i+(Is_{2N}(\t)-Ds_{2N}(\t))\j\Big],
\]
where
\begin{align*}
s_{2N}(\t)&=\frac{1}{2\pi}\frac{\sin(N\t/2)}{\sin(\t/2)},\\
\quad Is_{2N}(\t)&=\int_{0}^{\t}s_{2N}(\t')d\t',\quad \text{and}\quad 
Ds_{2N}(\t)=\frac{d}{d\t}s_{2N}(\t).
\end{align*}
\end{ex}

\begin{ex}[$\operatorname{Sine}_4$ process]\label{ex-2} The $\operatorname{Sine}_4$ process is the Pfaffian point process on $E=\R$ defined by the quaternion kernel~\cite{Dyson70} $K_{4}(x,y)=\sigma_4(x-y)$:
\[
\sigma_{4}(x)=\frac{1}{4}\left[s(x)- i (Is(x)+Ds(x))\i+(Is(x)-Ds(x))\j\right],
\]
where
\begin{align*}
s(x)&=\frac{\sin(\pi x)}{\pi x},\quad Is(x)=\int_{0}^{x}s(x')d x',\quad \text{and}\quad  Ds(x)=\frac{d}{dx}s(x). 
\end{align*}
The $\operatorname{Sine}_4$ can be thought of as a \emph{scaling limit} of the CSE process.
\end{ex}

\begin{ex}[Zeros of Gaussian power series]\label{ex-3}
The Gaussian power series \[
f(z) =\sum_{k=0}^\infty a_k z^k,
\] where $\{a_k\}$ are i.i.d. real standard Gaussian random variables,  defines almost surely a holomorphic function on the  open unit disk $\mathbb{D}$.  Matsumoto and Shirai \cite{Mat-Shirai}  showed that both the set of complex zeros and the set of real zeros of $f$ are Pfaffian point processes  on  the open unit disk $\mathbb{D}$ and on the interval $(-1, 1)$ respectively. 
\end{ex}

\begin{rmk}
In this paper, we consider $\H_\C$ since the complexified quaternions are generally involved in the correlation kernels of Pfaffian point processes as shown in the Examples \ref{ex-1}, \ref{ex-2} and \ref{ex-3} above. We also warn the reader that  the condition that all coefficients of the matricial kernel  (satisfying the assumption \eqref{anti-sym})
\[
\mathbb{K}(x, y) = \left[
\begin{array}{cc}
\mathbb{K}_{11}(x, y) & \mathbb{K}_{12}(x, y) 
\\
\mathbb{K}_{21}(x, y) & \mathbb{K}_{22}(x, y)
\end{array}
\right]
\]are real valued   in general does not  imply that the corresponding quaternion kernel is real-quaternionic. 
\end{rmk}

\subsection{Palm measures of Pfaffian point processes}

Let us briefly recall the definition of {\it Palm measures} of a point process. Here by Palm measures, we will always mean reduced Palm measures. We refer to \cite{DV-1,Kallenberg} for the general theory of Palm measures.

Let $\P$ be a simple point process on $E$ admitting $k$-th correlation measure $\lambda_k$ on $E^k$. Then for $\lambda_k$-almost every $k$-tuple $\mathfrak{q} = (q_1, \dots, q_k) \in E^k$, we  define  the  Palm measure   of $\P$ conditioned at $\mathfrak{q}$ as a point process $\P^{\mathfrak{q}}$ on $E$ by the following disintegration formula: for any non-negative Borel function $u: \Conf(E) \times E^k\rightarrow \R$, 
\[
\int\limits_{\Conf(E)}  \sum_{q_1, \dots, q_k \in X}^{*} u(X; \mathfrak{q}) \P(d X)  =    \int\limits_{E^k} \lambda_k(d\mathfrak{q}) \!\int\limits_{\Conf(E)} \!  u (X \cup \{q_1, \dots, q_k\};  \mathfrak{q})  \P^{\mathfrak{q}}(d X),
\]
where $\sum\limits^{*}$ denotes the sum over all mutually distinct points $q_1, \dots, q_k \in X$.

If we assume that the simple point process $\P$ admits correlation functions $\rho_m(x_1, \cdots, x_m)$ of all orders, then for $\lambda_m$-almost every $(x_1,\dots,x_m) \in E^m$,  the  Palm measure $\P^{x_1,\dots,x_m}$ exists  and admits all correlation functions $\rho_n^{x_1,\dots,x_m}$, $n\geq1$.  By \cite[Lemma 6.4]{Shirai03}, the correlation functions $\rho_n^{x_1,\dots,x_m}$ of the Palm measures $\P^{x_1,\dots,x_m}$ and the correlation functions $\rho_n$ of the original measure $\P$  satisfy the following  relation 
\begin{align}\label{eq:corr_Palm}
\rho_m(x_1,\dots,x_m)\rho_n^{x_1,\dots,x_m}(y_1,\dots,y_n)=\rho_{m+n}(x_1,\dots,x_m,y_1,\dots,y_n).
\end{align}

We now formulate  the Pfaffian analogue of the Shirai-Takahashi Theorem \cite[Theorem 6.5]{Shirai03} on Palm measures for determinantal point processes .
\begin{thm}\label{prop:Palm1}
Let $\P_K$ be a Pfaffian point process on $E$ induced by a self-adjoint quaternion kernel $K: E\times E \rightarrow \H_\C$. Then for $\lambda_1$-almost every $x_0\in E$, we have $K(x_0,x_0)>0$ and the Palm measure $\P_K^{x_0}$ coincides with the Pfaffian point process induced by the quaternion kernel
\begin{align}\label{eq:corr_Palm_Pf1}
\begin{split}
K^{x_0}(x,y)& =  \frac{1}{K(x_0,x_0)}\qdet
\left[ 
\begin{array}{ll}
K(x,y) & K(x_0,y)\\
K(x,x_0) & K(x_0,x_0)
\end{array} 
\right]
\\
& = 
K(x,y) - \frac{K(x,x_0)K(x_0,y)}{K(x_0,x_0)}.
\end{split}
\end{align}
In notation we have $\P_K^{x_0}=\P_{K^{x_0}}$. 
\end{thm}

\begin{cor}\label{cor-m-order-palm}
For each $m\geq1$ and $\lambda_m$-almost every $(x_1, \cdots, x_m)\in E^m$, we have $\qdet\left[K(x_i,x_j)\right]_{i,j=1}^m>0$ and the Palm measure $\P_K^{x_1,\dots,x_m}$ is a Pfaffian point process with correlation kernel $K^{x_1,\dots,x_m}$ given by
\begin{align}\label{eq:corr_Palm_gen}
K^{x_1,\dots,x_m} (x, y) = \frac{\qdet
\begin{bmatrix}
K(x,y) & K(x_1,y)& \cdots & K(x_m,y)\\
K(x,x_1) & K(x_1,x_1)& \cdots & K(x_m,x_1)\\
\vdots & &\ddots\\
K(x,x_m) & K(x_1,x_m) &\cdots &K(x_m,x_m)
\end{bmatrix}}{\qdet\left[K(x_i,x_j)\right]_{i,j=1}^m}.
\end{align}
 In notation, we have 
\begin{align}\label{eq-m-order-palm}
\P_K^{x_1,\dots,x_m}=\P_{K^{x_1,\dots,x_m}}.
\end{align}
\end{cor}
\begin{rmk}\label{rem-almost-sa}
Note that we write the kernel $K^{x_1,\dots,x_m}(x,y)$ as a ratio of quaternion determinants. The $m\times m$ quaternion matrix in the denominator is self-adjoint; the 
quaternion matrix in the numerator is almost self-adjoint. 
\end{rmk}

\subsection{Multiplicative functionals}
Suppose $g(x)$ is a complex-valued function on $E$ such that $g-1$ is compactly supported. Then $g(x)$ defines a multiplicative  functional $\Psi_g(X)$ on $\Conf(E)$  by the formula
\[
\Psi_g(X)=\prod_{x\in X}g(x),\quad X\in\Conf(E).
\]

We  prove that the product of a Pfaffian measure with a non-negative multiplicative functional on $\Conf(E)$ is, after normalisation, again a Pfaffian measure on $\Conf(E)$ induced by an operator that  can be written explicitly.  
In the determinantal case, preservation of the determinantal property under taking the product with a multiplicative functional is established in  \cite{Bufetov12}; see also \cite{Bufetov13, Bufetov18}.

Equip $\H_\C$ with the inner product 
\[
\langle q, q'\rangle = q_0 \bar{q_0'} + q_1 \bar{q_1'} + q_2 \bar{q_2'} + q_3 \bar{q_3'} \in \C.
\]
Then $\H_\C$ is a complex Hilbert space.  Let $L^2(E,\mu; \H_\C)$ be the Hilbert space~\cite{Viswanath71} of $\H_\C$-valued square integrable functions on $E$ with respect to the measure $\mu$.

For a measurable bounded non-negative function $g: E\rightarrow \R_{+}$, if  the operator $1 + (g-1)K$ on $L^2(E, \mu; \H_\C)$ is invertible, then we set 
\begin{align}\label{def-K-g}
K^g: = \sqrt{g}K (1 + (g-1)K)^{-1} \sqrt{g}. 
 \end{align}
Since $K$ is quaternion self-adjoint, it is easy to check, using the elementary identity $K(1 + (g -1)K)^{-1}= (1 + K(g-1))^{-1}K$,  that  $K^g$ is also quaternion self-adjoint.

\begin{lem}\label{lem-K-g}
Assume that $g(x)$ is a  non-negative measurable function on $E$ such that $g-1$ is compactly supported and  $\E_{\P_K} \Psi_g > 0$.  Then the operator $1 + (g -1)K$ on $L^2(E, \mu; \H_\C)$ is invertible and, in particular, the operator $K^g$ can be defined. 
\end{lem}

\begin{thm}\label{thm-mul-f}
Assume that $g(x)$ is a  non-negative measurable function on $E$ such that $g-1$ is compactly supported and $\E_{\P_K} \Psi_g > 0$.  Then the Pfaffian measure $\Psi_g \P_K$ on $\Conf(E)$, after normalization, is a Pfaffian measure with correlation kernel $K^g$. That is, we have 
\begin{align}\label{norm-mul-pp}
\frac{\Psi_g\P_K}{\E_{\P_K}\Psi_g}=\P_{K^g}.
\end{align}
\end{thm}

\subsection{Conditional measures}

We briefly recall the definition of conditional measures for point processes. 

Let $\P$ be a Borel probability measure on $\Conf(E)$.  Take a Borel subset $W\subset E$. Define the map $ \pi_W:  \Conf(E)  \rightarrow \Conf(W)$ by $\pi_W(X) = X\cap W$.  By  disintegrating the probability measure $\P$ with respect to the map $\pi_W$, for $(\pi_W)_{*}(\P)$-almost every configuration $X_{0}\in \Conf(W)$,  there exists a probability measure, denoted by $\P(\cdot| X_{0}, W)$,  supported on  the following fiber of $\pi_W$: 
\[
\pi_W^{-1}(X_0) =  \left\{Y \in \Conf(E)\,|\, Y \cap W  = X_0 \right\} = \left\{X_0 \sqcup Z\,|\,  Z \in \Conf(E\setminus W)\right\},
\]
  such that
\[
\P= \int_{\Conf(W)}  \P(\cdot| X_{0}, W)   (\pi_W)_{*}(\P) (  d X_{0}).
\]
Then, using the natural identification 
\[
\left\{X_0 \sqcup Z\,|\,  Z \in \Conf(E\setminus W)\right\}  \simeq \Conf(E\setminus W), 
\]
we also identify the probability measure $\P(\cdot| X_{0}, W)$ as a probability measure on the space $\Conf(E\setminus W)$ and  will be  referred to as the  conditional measure on $\Conf(E\setminus W)$ or conditional point process on $E\setminus W$ of the original point process $\P$, the condition being that the configuration on $W$ coincides with  the given $X_0$. In what follows, for $\P$-almost every configuration $X\in \Conf(E)$, we denote also
\[
\P (\cdot | X, W) = \P(\cdot |  X\cap W, W).
\]

We shall consider the conditional measures of Pfaffian point processes with respect to the conditioning that the configuration on a fixed   bounded measurable set is given, that is, we will take $W = B$ a bounded measurable set and consider the conditional measure 
\[
\P_K(\cdot | X, B)
\]
for a Pfaffian point process  $\P_K$ on $E$.

If $B\subset E$ is a bounded measurable set, then  for any simple configuration  $X\in\Conf(E)$, the subset $X\cap B \subset E$ is finite $X\cap B=\{x_1,\cdots,x_m\}$ and we will use the following notation:  
\[
\P_K^{X\cap B}=\P_K^{x_1,\ldots,x_m}, \quad  K^{X\cap B}=K^{x_1,\ldots,x_m}. 
\]

\begin{thm}\label{thm-cond-measure}
Let $B\subset E$ be a bounded Borel subset. Then, for $\P_K$-almost every $X\in\Conf(E)$, the operator $1 - \chi_BK^{X\cap B}$ is invertible and the conditional measure 
 $\P_K(\cdot|X,B)$ is again a  Pfaffian point process on $E\setminus B$  with a correlation kernel given by  
\[
K^{[X,B]} = \chi_{E \setminus B}K^{X\cap B}(1 - \chi_BK^{X\cap B})^{-1}\chi_{ E \setminus B}
\]
 In formula, for $\P_K$-almost every $X\in\Conf(E)$, we have 
\[
\P_K(\cdot|X,B) =\P_{K^{[X,B]}}.
\] 
\end{thm}


\section{Correlation kernels of the Palm measures}
In this section, we will prove Theorem \ref{prop:Palm1} and Corollary \ref{cor-m-order-palm}.

\begin{proof}[Proof of Theorem \ref{prop:Palm1}]
First, by the definition of Pfaffian point process, the first correlation measure is given by
\[
d\lambda_1(x) = K(x, x) d\mu(x). 
\]
Therefore, for $\lambda_1$-almost every $x \in E$,  we have $K(x, x) > 0$. Clearly, by the definition \eqref{eq:corr_Palm_Pf1} of the kernel $K^{x_0}$, we have 
\[
K^{x_0}(x,x)\in \R  \text{\, and \,} K^{x_0}(x,y)=\overline{K^{x_0}(y,x)}.
\] 
Hence $K^{x_0}$ defines a self-adjoint quaternion kernel. 
From \eqref{eq:corr_Palm} we have
$$
\rho_n^{x_0}(x_1,\dots,x_n)=\frac{1}{\rho_1(x_0)}\rho_{n+1}(x_0,x_1,\dots,x_n).
$$
\begin{claim} For all $n\geq1$, we have 
\be
\qdet\left[K^{x_0}(x_i,x_j)\right]_{i,j=1}^n=\frac{\qdet\left[K(x_i,x_j)\right]_{i,j=0}^n}{K(x_0,x_0)}.
\label{eq:Qdet_identity}
\ee
\end{claim}
If we accept the claim, we conclude that
$$
\rho_n^{x_0}(x_1,\dots,x_n)=\qdet\left[K^{x_0}(x_i,x_j)\right]_{i,j=1}^n,
$$
for all $n\geq1$.
It remains to prove the claimed identity \eqref{eq:Qdet_identity}.

To save writing we use the notation
$$
M_{ij}=K(x_i,x_j),\qquad M^0_{ij}=K^{x_0}(x_i,x_j)=M_{ij}-\frac{M_{i0}M_{0j}}{M_{00}}.
$$
Consider the matrix $(M_{ij})_{i,j=0}^n$ and perform the following sequence of column/row operations:   add the $0$-th column right-multiplied by $(-M_{01}/M_{00})$  to column $1$, add the $0$-th row left-multiplied by $(-M_{01}/M_{00})$  to row $1$; add the $0$-th column right-multiplied by $(-M_{02}/M_{00})$  to column $2$, add the $0$-th row left-multiplied by $(-M_{02}/M_{00})$  to row $2$; and so on. As a result we get the a self-adjoint matrix with block structure
\begin{multline}
M'=
\begin{bmatrix}
M_{00} &0 & 0 &\cdots &0\\
0 &M_{11}-\frac{M_{10}M_{01}}{M_{00}}  &M_{12}-\frac{M_{10}M_{02}}{M_{00}}  &\cdots &M_{1n}-\frac{M_{10}M_{0n}}{M_{00} }\\
0 &M_{11}-\frac{M_{20}M_{01}}{M_{00}}  &M_{22}-\frac{M_{20}M_{02}}{M_{00}}   &\cdots &M_{2n}-\frac{M_{20}M_{0n}}{M_{00}} \\
\vdots\\
0 &M_{n1}-\frac{M_{n0}M_{01}}{M_{00}} &M_{11}-\frac{M_{n0}M_{02}}{M_{00}}  &\cdots &M_{nn}-\frac{M_{n0}M_{0n}}{M_{00}} 
\end{bmatrix}\\
=
\left[
\begin{array}{c|c}
  M_{00}  & 0 \cdots 0 \\ \hline
  0  & \raisebox{-15pt}{{\large\mbox{{$(M^0_{ij})_{i,j=1}^n$}}}} \\[-4ex]
  \vdots & \\[-0.5ex]
  0  &
\end{array}
\right].
\end{multline}
By Lemma \ref{lem:e1} and Lemma \ref{lem:e2}, we conclude that
$$
\qdet [M_{ij}]_{i,j=0}^n=\qdet [M'_{ij}]_{i,j=0}^n=M_{00}\qdet [M^0_{ij}]_{i,j=1}^n.
$$
This is exactly the claim identity \eqref{eq:Qdet_identity}. 
\end{proof}

\begin{proof}[Proof of Corollary \ref{cor-m-order-palm}]
First, note that $K^{x_1,\dots,x_m}(x,y)=\overline{K^{x_1,\dots,x_m}(y,x)}$.  That is, the kernel $K^{x_1, \cdots, x_m}$ is quaternion self-adjoint. 
From~\eqref{eq:corr_Palm} we have ($n>m$):
$$
\rho_n^{x_0,x_1,\dots,x_{m-1}}(x_{m},\dots,x_{n})=\frac{\rho_{n+1}(x_0,x_1,\dots,x_n)}{\rho_{m}(x_0,x_1,\dots,x_{m-1})}.
$$
Therefore we need to show that
\begin{align}\label{eq:Qdet_identity2}
\qdet\left[K^{x_0,\dots,x_{m-1}}(x_i,x_j)\right]_{i,j=m}^{n}=\frac{\qdet\left[K(x_i,x_j)\right]_{i,j=0}^n}{\qdet\left[K(x_i,x_j)\right]_{i,j=0}^{m-1}}.
\end{align}
 Consider the measure  
 $$
 \P_K^{x_0,x_1\dots,x_m}=(\P_K^{x_0,\dots,x_{m-1}})^{x_{m}}=((\P_K^{x_0,\dots,x_{m-2}})^{x_{m}})^{x_{m-1}}=\cdots=(\cdots(\P_K)^{x_m}\cdots)^{x_{0}}.
 $$ 
 By Theorem \ref{prop:Palm1}, if $K^{x_1,\dots,x_m}(x_0,x_0)>0$, then $\P_K^{x_0,x_1,\dots,x_m}$ is a Pfaffian point process with correlation kernel
$$
K^{x_0,x_1\dots,x_m}(x,y)=(\cdots(K)^{x_m}\cdots)^{x_0}(x,y).
$$
(Note that the order of the points $x_0,x_1,\dots,x_m$ is immaterial in the above iteration.)
Therefore,
\begin{multline}\label{one-way}
\qdet\left[K(x_i,x_j)\right]_{i,j=0}^n=K(x_0,x_0)\qdet\left[K^{x_0}(x_i,x_j)\right]_{i,j=1}^n\\
=K(x_0,x_0)K^{x_0}(x_1,x_1)\qdet\left[K^{x_0,x_1}(x_i,x_j)\right]_{i,j=2}^n\\
\vdots\\
=K(x_0,x_0)K^{x_0}(x_1,x_1)\cdots K^{x_0,\dots,x_{m-2}}(x_{m-1},x_{m-1})\qdet\left[K^{x_0,\dots,x_{m-1}}(x_i,x_j)\right]_{i,j=m}^n.
\end{multline}
Set
\begin{align*}
M_{ij}&=K(x_i,x_j)\\
 M^0_{ij}&=K^{x_0}(x_i,x_j)\\
M^{01}_{ij}&=\left(K^{x_0}\right)^{x_1}(x_i,x_j)=K^{x_0,x_1}(x_i,x_j)\\
\vdots\\
M^{0\dots m-2}_{ij}&=\left(\cdots\left(K^{x_0}\right)^{x_1}\cdots\right)^{x_{m-2}}(x_i,x_j)=K^{x_0,\dots,x_{m-2}}(x_i,x_j).
\end{align*}
Now, by iterated  column/row operations we have
\begin{multline*}
\qdet\left(M_{ij}\right)_{i,j=0}^{m-1}=
\qdet
\left[
\begin{array}{c|c}
  M_{00}  & 0 \cdots  \\ \hline
  0  & \raisebox{-15pt}{{\large\mbox{{$(M^0_{ij})_{i,j=1}^{m-1}$}}}} \\[-4ex]
  \vdots & \\
    &
\end{array}
\right]\\
=
\qdet
\left[
\begin{array}{c|c|ccc}
  M_{00}  & 0 &0 &\cdots& \\ \hline
  0  & M^0_{11} &0 &\cdots&0\\\hline
0 & 0  & &\raisebox{-15pt}{{\large\mbox{{$(M^{01}_{ij})_{i,j=2}^{m-1}$}}}} \\[-4ex]
  \vdots &\vdots \\[-0.5ex]
  &
\end{array}
\right]\\
=\cdots=
\qdet
\begin{bmatrix}
M_{00} &0 &\cdots & &  &0\\
0 & M^{0}_{11}& 0&\cdots &  & \\
\vdots & 0& M^{01}_{22}&0 &\cdots  & \\
 &\vdots &0 & \ddots& & \\
  & & \vdots& & &\\
0   & & & & & M^{0\dots m-2}_{m-1,m-1} 
\end{bmatrix}.
\end{multline*}
Thus, 
\be\label{another-way}
\qdet\left[K(x_i,x_j)\right]_{i,j=0}^{m-1}=K(x_0,x_0)K^{x_0}(x_1,x_1)\cdots K^{x_0,x_1,\dots,x_{m-2}}(x_{m-1},x_{m-1}).
\ee
Therefore, combining \eqref{one-way} and \eqref{another-way}, we obtain 
\[
\qdet\left[K(x_i,x_j)\right]_{i,j=0}^n = \qdet\left[K(x_i,x_j)\right]_{i,j=0}^{m-1}  \qdet\left[K^{x_0,\dots,x_{m-1}}(x_i,x_j)\right]_{i,j=m}^n.
\]
This proves the desired equality \eqref{eq:Qdet_identity2}. 
\end{proof}

\section{Multiplicative functionals and conditional measures}
In what follows, let $K: E\times E\rightarrow \H_\C$ be a self-adjoint quaternion kernel which induces a bounded  locally trace-class operator on the Hilbert space $L^2(E, \mu; \H_\C)$. Suppose that $K$ induces a Pfaffian point process $\P_K$ on $E$. 

We will need the definition of quaternion Fredholm determinant and its relation with the usual Fredholm determinant, which are collected in  Appendix~\ref{app:Fredholmqdet}.

The following lemma  characterizes the Pfaffian point processes. 
\begin{lem}[{see  Rains \cite[Theorem 8.2]{Rains}}]\label{lem-mult-fl}
  Suppose that $g(x)$ is a measurable bounded function on $E$, and such that $g(x)-1$ is compactly supported.  Then 
\begin{align}\label{def-pf-pp}
\E_{\P_K}\Psi_g=\qdet(1+\sqrt{g-1}K\sqrt{g-1}).
\end{align}
Moreover, the identity \eqref{def-pf-pp} characterizes the Pfaffian point process $\P_K$.  
\end{lem}

\begin{proof}[Proof of Lemma \ref{lem-K-g}] 
By  \eqref{pf-det} and  \eqref{def-pf-pp} in Lemma \ref{lem-mult-fl},   the assumption $\E_{\P_K} \Psi_g > 0$ implies 
\[
\det\left( 1 + \varphi \left( \sqrt{g-1}K\sqrt{g-1} \right)  \right) > 0.
\]
Therefore, by the classical theory of Fredholm determinants, the operator 
$$
1 + \varphi \left( \sqrt{g-1}K\sqrt{g-1} \right)  =  \varphi \left( 1 +  \sqrt{g-1}K\sqrt{g-1} \right) 
$$
is invertible.  This implies that $1 + \sqrt{g-1}K \sqrt{g-1}$ is invertible and hence so is $1 + (g-1)K$. 
\end{proof}

\begin{proof}[Proof of Theorem \ref{thm-mul-f}]
Take any measurable bounded function $h(x)$ on $E$ such that $g-1$ is compactly supported. By \eqref{def-pf-pp}, we have 
\[
 \int_{\Conf(E)}\Psi_h \Psi_g d\P_K =   \E_{\P_K} (\Psi_{gh}) =  \qdet \Big( 1 + \sqrt{gh-1} K \sqrt{gh-1}\Big). 
\]
Note that for any complex-valued functions $g_1, g_2$ and quaternion kernel $T$ on $E$,
\begin{align}\label{bi-module-st}
\varphi( g_1 T g_2) = g_1 \varphi(T) g_2
\end{align}
 Therefore,  by \eqref{pf-det}, we have 
\begin{align*}
\left( \qdet ( 1 + \sqrt{gh-1} K \sqrt{gh-1})\right)^2  &= \det\left( 1 + \varphi(\sqrt{gh-1} K \sqrt{gh-1})\right)  
\\
&= \det ( 1 +\sqrt{gh-1} \varphi(K) \sqrt{gh-1}) \\
&   =  \det( 1 + \sqrt{gh-1} \varphi(K) \sqrt{gh-1}). 
\end{align*}
Using the regularization of the Fredholm determinant in \cite[Section 2.4]{Bufetov13}, we have 
\begin{align}\label{fred-AB}
\det( 1 + \sqrt{gh-1}  \varphi(K) \sqrt{gh-1}) = \det( 1 + (gh-1)  \varphi(K)). 
\end{align}
Observe that we have 
\begin{align*}
1 + (gh -1) K &=       \Big( 1+ (gh-1)K\Big)\Big(1 + (g-1)K\Big)^{-1}  \Big( 1    + (g-1)K\Big)
\\
&= \Big( 1 + (g - 1)K + (h-1)gK \Big)\Big(1 + (g-1)K\Big)^{-1}  \Big( 1    + (g-1)K\Big)
\\
&= \Big( 1 + (h-1)  gK (1 + (g-1)K)^{-1} \Big) \Big( 1 + (g-1)K\Big).  
\end{align*}
Therefore, again by using \eqref{bi-module-st},  we have 
\[
1 + (gh -1) \varphi( K )=  \Big[ 1 + (h-1)  \varphi\Big( g K (1 + (g-1)K)^{-1} \Big) \Big] \Big[ 1 + (g-1)\varphi(K)\Big]. 
\]
 Applying the multiplicativity of the  ordinary Fredholm determinant, we have 
\begin{multline*}
\left( \qdet ( 1 + \sqrt{gh-1} K \sqrt{gh-1})\right)^2 = \det ( 1 + (gh-1) \varphi( K))  =
\\
= \det\left[ 1 + (h-1)   \varphi\left ( g  K ( 1 + (g-1)  K )^{-1} \right)\right]  \det ( 1 + (g-1) \varphi(K)).
\end{multline*}
Then by using \eqref{bi-module-st}, similar identity as \eqref{fred-AB} and the definition \eqref{def-K-g} of $K^g$, we obtain 
\begin{multline}\label{AB-sq-C-sq}
\left( \qdet ( 1 + \sqrt{gh-1} K \sqrt{gh-1})\right)^2 =
\\
=   \det\left[(1  +     \varphi\left ( \sqrt{h-1} K^g \sqrt{h-1}\right) \right)  \det \left(1 +  \varphi( \sqrt{g-1} K \sqrt{g-1})\right)
\\
=  \left( \qdet ( 1 + \sqrt{h-1} K^g \sqrt{h-1})\right)^2   \left(\qdet ( 1 +   \sqrt{g-1} K \sqrt{g-1})\right)^2,
\end{multline}
where we used~\eqref{pf-det}.
Now for any $z\in \C$, set
\[
 h_z(x) = 1 + z ( h(x) -1).
\]
By  Lemma  \ref{lem-holo}, the following two functions are holomorphic on $\C$: 
\[
f_1(z)= \qdet ( 1 + \sqrt{gh_z-1} K \sqrt{gh_z-1}), \quad f_2(z) =  \qdet ( 1 + \sqrt{h_z-1} K^g \sqrt{h_z-1}). 
\]
Substituting $h_z$ into the equality \eqref{AB-sq-C-sq}, we obtain
\[
f_1(z)^2 = f_2(z)^2  \left(\qdet ( 1 +   \sqrt{g-1} K \sqrt{g-1})\right)^2, \quad z \in \C. 
\] 
Since $f_1, f_2$ are entire functions, there exists a constant $C\in \{1, -1\}$, such that 
\begin{align}\label{id-pm}
f_1(z) =  C f_2(z)  \qdet ( 1 +   \sqrt{g-1} K \sqrt{g-1}).
\end{align}
Note that $h_0(x) = 1$ and $h_1(x)=h(x)$. Hence 
\[
f_1(0)  = \qdet ( 1 + \sqrt{g-1} K \sqrt{g-1})  = \E_{\P_K} \Psi_g > 0, \quad f_2(0) = 1. 
\]
This implies that the constant $C$ in the equality \eqref{id-pm} must be equal to $1$, and substituting $z = 1$ into \eqref{id-pm}, we obtain 
\begin{multline}\label{pf-prod-pf}
\qdet ( 1 + \sqrt{gh-1} K \sqrt{gh-1})=
\\
=   \qdet ( 1 + \sqrt{h-1} K^g \sqrt{h-1})  \qdet ( 1 +   \sqrt{g-1} K \sqrt{g-1}).
\end{multline}
The equality \eqref{pf-prod-pf} now can be recast as 
\begin{multline}\label{frac-int}
\frac{\displaystyle \int_{\Conf(E)}\Psi_h \Psi_g d\P_K}{\displaystyle \int_{\Conf(E)} \Psi_g d\P_K} = \frac{\qdet ( 1 + \sqrt{gh-1} K \sqrt{gh-1})}{\qdet ( 1 + \sqrt{g-1} K \sqrt{g-1})}  = 
\\
= \qdet ( 1 + \sqrt{h-1} K^g \sqrt{h-1}). 
\end{multline}
This, combined with the characterization, Lemma \ref{lem-mult-fl},  of Pfaffian point processes, implies the desired conclusion \eqref{norm-mul-pp}. 
\end{proof}

We now proceed to the proof of  Theorem \ref{thm-cond-measure}.   First let us recall some general results for point processes. 
For a Borel subset $W\subset E$,  define 
\[
\Conf(E; W) = \left\{X\in\Conf(E)\,|\,  X\subset W\right\}.
\]
We have the following natural identification  
\begin{align}\label{id-conf-E-W}
\Conf(E; W)  \xrightarrow[\simeq]{\quad \iota_W \quad } \Conf(W). 
\end{align}

For a Borel probability measure $\P$ on $\Conf(E)$ and a Borel subset $W$, if $\P(\Conf(E; W)) > 0$, then we define  the measure $\overline{\P|_{\Conf(W)}}$ to be the probability measure on $\Conf(W)$, which corresponds, using the identification \eqref{id-conf-E-W}, to the  normalized restriction of $\P$ on the subset $\Conf(E; W) \subset \Conf(E)$. That is, 
 \[
\overline{\P|_{\Conf(W)}} = ({\iota_W})_{*} \left( \overline{\P|_{\Conf(E; W)}} \right)=
(\iota_W)_{*} \left( \dfrac{\P |_{\Conf(E; W)}}{\P(\Conf(E; W))}\right). 
\]

\begin{rmk}
We warn the reader the measure $\overline{\P|_{\Conf(W)}}$ is different from the pushforward measure $(\pi_W)_{*}(\P)$ of $\P$ under the `forgetting' map  $\pi_W$. 
\end{rmk}

We will use the following general results on point processes.
\begin{prop}[{\cite[Lemma 6.2]{Bufetov19}}]\label{lem-good-rel}
Let $\P$ be a point process on $E$.  Then for any bounded Borel subset $B\subset E$, we have
\begin{align}\label{non-vanish-con}
\P\left(\#_B = \#(X \cap B)\Big| X,  E\setminus B\right) > 0 \quad \text{for $\P$-almost every $X \in \Conf(E)$}.
\end{align}
\end{prop}

\begin{prop}[{\cite[Proposition 8.1]{Bufetov19}}]\label{prop-cond-m}
Let $B\subset E$ be a bounded Borel subset. If $\P$ is a simple point process on $E$ admitting correlation functions of all orders, then for $\P$-almost every $X \in \Conf(E)$, we have $\P^{X\cap B} ( X\cap B = \emptyset) > 0$ and 
\begin{align}\label{eq-cond-m}
\P(\cdot|X,B) =\overline{\P^{X\cap B}|_{\Conf(E\setminus B)}}.
\end{align}
\end{prop}

\begin{proof}[Proof of Theorem \ref{thm-cond-measure}]
By Corollary \ref{cor-m-order-palm} and  Proposition \ref{prop-cond-m},  for $\P$-a.e. $X\in\Conf(E)$, 
\begin{align}\label{qdet-pos}
\begin{split}
 &\P_K^{X\cap B} ( X\cap B = \emptyset)  = \P_{K^{X\cap B}} ( X\cap B = \emptyset) \\
&=\E_{\P_{K^{X\cap B}}} \left(  \Psi_{\chi_{E\setminus B}}\right)  = \qdet \left( 1  - \chi_B K^{X\cap  B} \chi_B\right) >0.
\end{split}
\end{align}
Similar to the proof of the invertibility statement in Lemma \ref{lem-K-g}, the non-vanishing statement \eqref{qdet-pos} implies that $1 - \chi_B K^{X\cap B} \chi_B$ is invertible and thus so is the operator $1 - \chi_B K^{X\cap B}$. 

Now applying  \eqref{eq-cond-m} to $\P_K$, and then applying \eqref{eq-m-order-palm}-\eqref{norm-mul-pp},  we obtain 
\begin{multline*}
\P_K(\cdot|X,B) =\overline{\P_K^{X\cap B}|_{\Conf(E\setminus B)}} =\overline{\P_{K^{X\cap B}}|_{\Conf(E\setminus B)}} = 
\\
= \frac{ \Psi_{\chi_{E\setminus B} } \cdot  \P_{K^{X\cap B}}}{\E_{\P_{K^{X\cap B}}} \left(\Psi_{\chi_{E\setminus B} }\right)}  =\P_{\chi_{E\setminus B} K^{X\cap B}  ( 1 - \chi_B K^{X\cap B})^{-1} \chi_{E\setminus B}} = \P_{K^{[X, \, B]}}.
\end{multline*}
This completes the proof of the theorem. 
\end{proof}

{\bf {Acknowledgements}}
AB’s research is supported by the European Research Council (ERC) under the European Union Horizon 2020 research and innovation programme, grant 647133 (ICHAOS), by the Agence Nationale de Recherche, project ANR-18-CE40-0035, and by the Russian Foundation for Basic Research, grant 18-31-20031. The research of FDC is supported by  the European Research Council (Grant Number 669306) and partially supported by Gruppo Nazionale di Fisica Matematica  GNFM-INdAM.  The research of YQ is supported by grants NSFC Y7116335K1,  NSFC 11801547 and NSFC 11688101 of National Natural Science Foundation of China. FDC would also like to thank the organisers of the research school `Gaz de Coulomb, int\'egrabilit\'e et \'equations de Painlev\'e 2019' at CIRM - Marseille, where this work was started.

\appendix
 \section{Quaternion determinants}
 \label{app:qdet}
 In the investigation of eigenvalues of random
matrices, it was found by Dyson~\cite{Dyson70} that the correlation functions can be
conveniently expressed as determinants of quaternion matrices, constructed according
to the rules laid down by Moore~\cite{Moore22}.

In view of the non-commutativity of multiplication of quaternions a  canonical definition of the determinant of an unconditioned quaternion matrix is an open problem (see~~\cite{Aslaksen96,Kargin14} for a survey of the most known attempts). However, as argued by Moore~\cite{Moore22}, for a {\it self-adjoint} quaternion matrix $M$ a useful determinant, $\qdet M$,
is definable by suitable prescription of the order
of the $n$ factors of the $n!$ terms of the formal ordinary determinant. It turns out that Moore's definition of determinants applies not only to self-adjoint
matrices but also to matrices which are self-adjoint except for a single row or column.

\begin{rmk}\label{rem-sa}
The following remarks will clarify the quaternion self-adjointness:
\begin{itemize}
\item[(i)] If $M$ is a self-adjoint quaternion matrix, then for any scalar $z \in \C$, the matrix $z M$ is again  a self-adjoint quaternion matrix.
\item[(ii)] Any  matrix over $\C$, when viewed as a quaternion matrix over $\H_\C$, is quaternion self-adjoint  iff it is {\it symmetric} in the usual sense.  Therefore, a Hermitian matrix over $\C$ is not necessarily quaternion self-adjoint. 
\end{itemize}
\end{rmk}

 Following Dyson \cite{Dyson72}, we define the quaternion determinant $\qdet M$ of an almost self-adjoint $n\times n$ quaternion matrix $M$ by recursion on $n$ as follows. 
\begin{dfn}
\label{dfn:det}
Let $M$ be an $n\times n$ almost self-adjoint quaternion matrix. For $n=1$ we set
\be
\qdet M=M_{11},
\ee
and for $n>1$ we set
\be
\qdet M=\sum_{l=1}^k\epsilon_{kl}M_{kl}\qdet M(k,l),
\label{eq:rec_qdet}
\ee
where $k$ is the integer singled out by the condition~\eqref{eq:a-s-a}, 
\be
\epsilon_{kl}=-1,\quad \text{if $k\neq l$};\quad \epsilon_{kk}=+1,
\ee
and $M(k,l)$ is the $(n-1)\times(n-1)$ almost self-adjoint matrix obtained from $M$ by first replacing the $l$-th column by the $k$-th column, and then deleting both the $k$-th column and the $k$-th row. 
\end{dfn}
The above algorithm agrees with the definition of the ordinary determinant for commuting variables. The next theorem asserts that for self-adjoint matrices the recursion \eqref{eq:rec_qdet} is independent of $k$ and thus the quaternion determinant is uniquely defined. Moreover, the quaternion determinant enjoys many of the useful properties of the ordinary determinant.
\begin{thm}[Dyson \cite{Dyson70,Dyson72}] 
\label{thm:Dyson_prop}
Let $M$ be an almost self-adjoint quaternion matrix. 
\begin{enumerate} 
\item[(i)] If $M$ is self-adjoint, then the value of $\qdet M$ defined by \eqref{eq:a-s-a} is independent of $k$ and   $\qdet M=\overline{\qdet M} \in\C$.
\item[(ii)] If $M$ has two identical columns (or rows) then $\qdet M=0$.
 \item[(iii)] Let $M^{(1)}$, $M^{(2)}$ be almost self-adjoint matrices, satisfying the conditions~\eqref{eq:a-s-a}, and
 \begin{align}
M^{(1)}_{ij}=M^{(2)}_{ij}=M_{ij},\quad i\neq k\\
M^{(1)}+M^{(2)}_{ij}=M_{ij},\quad i= k
 \end{align} 
 for some common row-index $k$. Then $\qdet M^{(1)}+\qdet M^{(2)}=\qdet M$.
 \end{enumerate}
\end{thm}

\begin{rmk}
The multiplicative property $\det(AB)=\det(A)\det(B)$ of the ordinary determinant does not have an analogue in the quaternion determinant setting. The product of two almost self-adjoint matrices $A$ and $B$ is in general neither self-adjoint nor almost self-adjoint, and so $\qdet(AB)$ is generally undefined.
\end{rmk}

As suggested by Moore~\cite{Moore22}, a quaternion determinant of a self-adjoint matrix can be defined by the combinatorial formula that extends Cayley's definition of ordinary determinant with a careful prescription on the order of the $n$ factors of the $n!$ terms. Namely, let $S_n$ be the group of permutations of the set $\{1, \cdots, n\}$.   Write every permutation $\sigma \in S_n$ uniquely as a product of cycles (cycles of length one are also considered): 
\begin{align}\label{cyc-sigma}
\sigma = (n_1 i_2\dots i_s) (n_2 j_2\dots j_t)\dots (n_r k_2 \dots k_l),
\end{align}
where $n_i$ are largest elements of each cycle and $n_1 > n_2> \cdots> n_r$. For an $n\times n$ self-adjoint quaternion matrix $M$ and an element $\sigma \in S_n$ written as in \eqref{cyc-sigma},  we set 
\[
M_\sigma  =  (M_{n_1 i_2} M_{i_2 i_3} \cdots M_{i_s n_1})\cdots (M_{n_r k_2} M_{k_2 k_3} \cdots M_{k_l n_r}).
\]
Then, as proved by Dyson \cite{Dyson72}, the following original definition by Moore of quaternion determinant of self-adjoint matrices is equivalent to the recursion in Definition \ref{dfn:det}.

\begin{prop} If $M$ is an $n\times n$ self-adjoint quaternion matrix, then 
\be
\qdet M=\sum_{\sigma\in S_n}\epsilon(\sigma) M_\sigma,
\label{eq:def_qdet}
\ee
where $\epsilon (\sigma)$ is the sign  of the permutation $\sigma$. 
\end{prop}

\begin{ex} 
For a $3\times3$ quaternion self-adjoint matrix $M$, we have 
$
\qdet M=M_{33}M_{22}M_{11}-M_{33}M_{21}M_{12}-M_{31}M_{13}M_{22}-M_{32}M_{23}M_{11}+M_{31}M_{12}M_{23}+M_{32}M_{21}M_{31}.
$
\end{ex}

 The complex algebra $\H_\C$ is isomorphic to the complex algebra of $2\times 2$ matrices over $\C$ using the following correspondence rules:
\[
\i = 
\left[\begin{array}{cc}
0 & i
\\
i & 0
\end{array}
\right], 
\quad 
\j = 
\left[
\begin{array}{cc}
0 & -1
\\
1 & 0
\end{array}
\right], 
\quad 
\k = 
\left[
\begin{array}{cc}
i & 0
\\
0 & -i 
\end{array}
\right].
\]
Therefore, the map sends $q=q_0+q_1\i+q_2\j+q_3\k$ to
\begin{align}\label{def-phi-q}
\varphi(q)=\left[\begin{array}{cc}
 q_0+iq_3 & i q_1 - q_2 \\
 iq_1 + q_2 &  q_0-iq_3
 \end{array}
\right]\in M_2(\C),
\end{align}
and  the involution of  taking conjugate is defined by 
\begin{align}\label{matrix-conj}
\left[
\begin{array}{cc}
x & y 
\\
z&w
\end{array}
\right]^\dagger
= 
\left[
\begin{array}{cc}
w &  -y 
\\
-z& x
\end{array}
\right].
\end{align}
Using the above correspondence, we may map each $n\times n$ matrix $M$ over $\H_\C$ to a $2n \times 2n$ matrix $\varphi(M)$ with complex entries by replacing each coefficient of $M$ by a $2 \times 2$ block of elements in $\varphi(M)$. Note that for any integer $n\ge 1$,  the map 
\begin{align}\label{def-varphi}
\varphi: M_n(\H_\C) \rightarrow M_n (M_2(\C))
\end{align}
is an isomorphism between complex algebras.  
We have 
\begin{align}\label{def-mat-conj}
\varphi(M^\dagger) = Y_n \varphi(M)^T Y_n^{-1}, 
\end{align}
where $T$ denotes transposition and $Y_n$ is the block diagonal matrix 
\[
Y_n =  \varphi ( \operatorname{diag} (\underbrace{\j, \cdots, \j}_{\text{$n$ times}}))=\operatorname{diag}\left(\underbrace{\left[
\begin{array}{cc}
0 &-1
\\
1 & 0
\end{array}
\right], \cdots, \left[
\begin{array}{cc}
0 &-1
\\
1 & 0
\end{array}
\right]}_{\text{$n$ times}}\right).
\]
Note that $\det Y_n=1$, and
\begin{align}\label{inv-Y-n}
Y_n^{-1} = - Y_n.
\end{align}

For self-adjoint quaternion matrices there is another widely used representation of the quaternion determinant as \emph{Pfaffian} of an associated antisymmetric complex matrix. More precisely, for any $n\times n$ quaternion matrix,  define $\psi(M): = -Y_n \varphi (M)$. Then we have 
\[
\psi(M^\dagger) = - \psi(M)^T.
\]
In particular, the $n\times n$ quaternion matrix $M$ is self-adjoint if and only if the $2n \times 2n$ complex matrix $\psi(M)$ is {\it skew-symmetric} and thus the  Pfaffian $\Pf(\psi(M))$ exists (see  Dyson \cite{Dyson70}) and  by Dyson \cite[formula (4.7)]{Dyson72},   we have 
\begin{align}\label{qdet-pf}
\qdet (M)  =\Pf(\psi(M)) = \Pf(-Y_n\varphi(M)). 
\end{align}
In particular, $\det Y_n=\det(-Y_n) = 1$, we have the identity
\begin{align}\label{sq-qdet}
[\qdet (M)]^2 = \det(-Y_n \varphi(M)) = \det(-Y_n) \det(\varphi(M)) = \det(\varphi(M)). 
\end{align}

A useful consequence of this identity is the following 
\begin{lem}\label{lem-inverse}
Let $M$ be a self-adjoint $n\times n$ quaternion matrix. Then $\qdet(M) \ne 0$ if and only if $M$ is invertible. 
\end{lem}

\begin{proof}
The identity \eqref{sq-qdet} implies that $\qdet(M) \ne 0$ if and only if the matrix $\varphi(M)$ is invertible.  The assertion of the lemma follows since $\varphi$ is an isomorphism between  the two complex algebras $M_n(\H_\C)$ and $M_n (M_2(\C))$. 
\end{proof}

In what follows,  we will repeatedly make use of the following elementary facts. Recall first an elementary identity for Pfaffian:  consider a commutative ring, then for any $2n\times 2n$ skew-symmetric matrix $A$ and an arbitrary $2n \times 2n$ matrix $B$ over the commutative ring, we have   
\begin{align}\label{pf-id}
\Pf (B A B^T) = \det (B) \Pf(A). 
\end{align}
\begin{lem}[Elementary row/column operations]
\label{lem:e1} Let $M$ be an $n\times n$ self-adjoint quaternion matrix.  Let $M'$ be the self-adjoint matrix obtained from $M$
\begin{enumerate}
\item[(i)]  by switching the column $i$ with column $j$, and row $i$ with row $j$. Then $\qdet M'=\qdet M$.
\item[(ii)]   by right-multiplying column $i$ by $q\in\H_\C$ and left-multiplying row $i$ by $\overline{q}$. Then $\qdet M'=q\overline{q}\qdet M$.
\item[(iii)]   by adding a right-multiple  of column $i$ by $q\in \H_\C$ to column $j$ and a left-multiple of row $i$ by $\bar{q}$   to row $j$. Then $\qdet M'=\qdet M$. 
  \end{enumerate}
\end{lem}
\begin{proof} We will make use of the isomorphism $\varphi$  and its properties. 
\par
(i) Let $P_{ij}$ be the $n\times n$ matrix corresponding to the transposition $(ij)\in S_n$. Then $M'= P_{ij}M P_{ij}$.
We have 
\[
P_{ij}^{\dagger} = P_{ij}, \quad Y_n \varphi(P_{ij}) = \varphi(P_{ij})Y_n. 
\]
Hence $\varphi(P_{ij})^T = \varphi(P_{ij})$. Therefore,  by \eqref{qdet-pf} and \eqref{pf-id}, we have 
\begin{multline*}
\qdet M' = \Pf(-Y_n \varphi(M')) = \Pf(-Y_n \varphi(P_{ij}MP_{ij}))   = 
\\
= \Pf(-Y_n \varphi(P_{ij})\varphi(M) = \Pf\Big(\varphi(P_{ij})   [-Y_n \varphi(M)] \varphi(P_{ij})^T\Big)=
\\
 = \det(\varphi(P_{ij})) \Pf(-Y_n \varphi(M)) =     \Pf(-Y_n \varphi(M)) = \qdet M,
\end{multline*}
where we used the elementary equality $\det(\varphi(P_{ij})) = 1$. 
\par
(ii) Given $q \in \H_\C$, define a diagonal quaternion matrix of size $n\times n$ by 
$$
D=\operatorname{diag}(1, \ldots, 1, q, 1, \ldots, 1).
$$
where $q$ is in the $i$th position.
Then, $M' = D^\dagger M D$. 
Using  \eqref{def-mat-conj}, we have 
\[
\varphi(D^\dagger) = Y_n \varphi(D)^T Y_n^{-1}.
\]
Therefore, by  \eqref{inv-Y-n}, \eqref{qdet-pf} and \eqref{pf-id}, we obtain 
\begin{multline*}
\qdet M'= \Pf(-Y_n \varphi(D^\dagger MD)) =    \Pf(-Y_n \varphi(D^\dagger) \varphi(M) \varphi(D)) = 
\\
 = \Pf\left(\varphi(D)^T [-Y_n \varphi(M)] \varphi(D)\right) =  \det \varphi(D) \Pf(-Y_n \varphi(M)) = \det\varphi(q) \qdet M. 
\end{multline*}
From~\eqref{def-phi-q}, we see that 
\[
\det\varphi(q) = q_0^2 + q_1^2   + q_2^2+q_3^2 = q \bar{q} \in \C. 
\]
Hence we obtain the desired equality $\qdet M' = q\bar{q} \qdet M$.
\par
(iii) Set $E_{ij}$ the elementary matrix with $(i,j)$-entry $1$ and all other entries $0$. For any $q\in \H_\C$, define an $n\times n$ quaternion matrix 
\[
A = I_n  + q E_{ij},
\]
where $I_n$ the the identity matrix of size $n\times n$.  Then the operation in (iii) on the matrix $M$ is given by 
\[
M' = A^\dagger M A. 
\]
Using  \eqref{def-mat-conj},  \eqref{inv-Y-n}, \eqref{qdet-pf} and \eqref{pf-id}, we obtain 
\begin{multline*}
\qdet M' = \Pf (-Y_n \varphi( A^\dagger M A)) = \Pf(-Y_n \varphi(A^\dagger) \varphi(M) \varphi(A)) = \\
= \Pf\Big( \varphi(A)^T [-Y_n \varphi(M)] \varphi(A)\Big) = \det \varphi(A) \Pf(-Y_n \varphi(M)) = \qdet M,
\end{multline*}
where we used the elementary equality $\det\varphi(A) = 1$. 
\end{proof}
\begin{lem}
\label{lem:e2} Let $M^{(1)}$ and $M^{(2)}$ be self-adjoint quaternion matrices of size $n_1$ and $n_2$, respectively. Consider the $(n_1+n_2)\times(n_1+n_2)$ self-adjoint block matrix 
$$
M=M^{(1)}\oplus M^{(2)}=
\left[\begin{array}{cc}
M^{(1)} &0\\
0& M^{(2)}
\end{array}
\right].
$$
Then, $\qdet M=\qdet M^{(1)}\qdet M^{(2)}$.
\end{lem}
\begin{proof}
A consequence of the combinatorial formula~\eqref{eq:def_qdet}.
\end{proof}

\begin{thm}[Dyson \cite{Dyson72}] If $M$ is almost self-adjoint and has two identical columns (or rows) then $\qdet M=0$.
 Let $M$, $N$, $P$  be almost self-adjoint matrices, satisfying the conditions~\eqref{eq:a-s-a}, and
 \begin{align}
 M_{ij}=N_{ij}=P_{ij},\quad i\neq k\\
 M_{ij}+N_{ij}=P_{ij},\quad i= k
 \end{align} 
 for some common row-index $k$. Then $\qdet M+\qdet N=\qdet P$.
\end{thm}

\section{Fredholm quaternion determinant}
\label{app:Fredholmqdet}
Let $T\colon E\times E\rightarrow \H_\C$ be a self-adjoint quaternion kernel $T(x,y)=\overline{T(y,x)}$, such that  $T$ induces a bounded {\it trace-class} operator, denoted again by $T$, on the Hilbert space $L^2(E, \mu; \H_\C)$ via the formula 
\[
T f (x)= \int_E   T(x, y) f(y) d\mu(y), \quad  \text{for all $  f \in L^2(E, \mu; \H_\C)$},
\]
where $T(x, y) f(y)$ is understood as the multiplication of two elements $T(x, y)$ and $f(y)$ in the algebra $\H_\C$.  Then, we define (see \cite{Kargin14}) the Fredholm quaternion determinant of $1 + T$ by
\begin{align}\label{def-fred-qdet}
\qdet(1+T) =1+\sum_{n=1}^{\infty}\frac{1}{n!}\int_{E^n}\qdet\left[T(x_i,x_j)\right]_{i,j=1}^nd\mu(x_1)\cdots d\mu(x_n).
\end{align}
Note that, if $E$ is finite, so that $1+T$ is a matrix, then this definition agrees with the usual definition of the quaternion determinant of the matrix $1+T$.  Since $T$ is quaternion self-adjoint,  by Remark \ref{rem-sa}, for any complex number $z \in \C$,  the kernel $z T$ is also quaternion self-adjoint and the Fredholm quaternion determinant $\qdet (1 + z T)$ is then defined.  Remark \ref{rem-sa} implies also that for any bounded complex valued function $f(x)$ on $E$, the kernel $\sqrt{f} T \sqrt{f}$ is well-defined (does not depend on the choices of the branch of the multi-valued function $z \mapsto \sqrt{z}$) and is quaternion self-adjoint.

\begin{lem}\label{lem-holo}
Let $T: E\times E\rightarrow \H_\C$ be a self-adjoint quaternion kernel  which induces a bounded  trace-class operator on the Hilbert space $L^2(E, \mu; \H_C)$. Assume that for any $z\in \C$, we have a bounded measurable function $f_z: E\rightarrow \C$ and the map $\C \ni z \mapsto f_z \in L^\infty(E,\mu)$ is holomorphic on the whole complex plane.  Then 
the function 
\[
\C \ni z \mapsto \qdet (1 + \sqrt{f_z} T \sqrt{f_z})
\]
is also holomorphic on the whole complex plane. 
\end{lem}
\begin{proof}
Note that for any $n\ge 1$ and $x_1, \cdots, x_n\in E$, by item (ii) of Lemma \ref{lem:e1}, we have 
\[
\qdet\left[ \sqrt{f_z(x_i)}T(x_i,x_j) \sqrt{f_z(x_j)}\right]_{i,j=1}^n =  \prod_{i = 1}^n f_z(x_i)   \qdet\left[ T(x_i,x_j) \right]_{i,j=1}^n.
\]
Now the lemma follows from the definition \eqref{def-fred-qdet} and the following equalities: 
\begin{multline*}
\qdet(1+ \sqrt{f_z}T \sqrt{f_z}) =
\\
= 1+\sum_{n=1}^{\infty}\frac{1}{n!}\int_{E^n}\qdet\left[ \sqrt{f_z(x_i)}T(x_i,x_j) \sqrt{f_z(x_j)}\right]_{i,j=1}^nd\mu(x_1)\cdots d\mu(x_n)  = 
\\
= 1+\sum_{n=1}^{\infty}\frac{1}{n!}\int_{E^n}  \prod_{i = 1}^n f_z(x_i)   \qdet\left[T(x_i,x_j) \right]_{i,j=1}^n  d\mu(x_1)\cdots d\mu(x_n).
\end{multline*}
\end{proof}

 We shall also use the following  identity   (see \cite[Lemma 8.1]{Rains}) which is an extension  to the case of the Fredholm quaternion determinant of the identity \eqref{sq-qdet}. Recall that we have an isomorphism of complex algebras: $\varphi: \H_\C\rightarrow M_2(\C)$. To the quaternion kernel $T$, we define a matrix-valued kernel $\varphi(T)$ by setting
\[
\varphi(T)(x, y) = \varphi(T(x, y)).
\]
In particular, we can view the kernel $\varphi(T)$ as the integral kernel of an integral operator, denoted again by $\varphi(T)$, acting on the Hilbert space 
\[
L^2(E, \mu; \C^2) = L^2(E, \mu) \oplus L^2(E, \mu).
\]
Then, we have the identity
\begin{align}\label{pf-det}
\left[\qdet (1 + T)\right]^2 = \det( 1 +  \varphi(T)).
\end{align}

\end{document}